\documentclass[10pt]{amsart}
\usepackage[utf8]{inputenc}

\usepackage{amssymb,amsthm,amsmath}
\usepackage{mathrsfs}
\usepackage{enumerate}
\usepackage{graphicx,color}
\usepackage[hidelinks]{hyperref}

\newcommand{\E}{\mathbb{E}}
\newcommand{\1}{\textbf{1}}
\newcommand{\R}{\mathbb{R}}

\newcommand{\p}[1]{\mathbb{P}\left( #1 \right)}

\newcommand{\red}{}

\DeclareMathOperator{\Var}{Var}

\usepackage[paper=a4paper, left=1.5in, right=1.5in, top=1.6in, bottom=1.4in]{geometry}
\newtheorem{theorem}{Theorem}
\newtheorem{lemma}[theorem]{Lemma}
\newtheorem{corollary}[theorem]{Corollary}

\theoremstyle{remark}
\newtheorem{remark}[theorem]{Remark}

\theoremstyle{definition}

\title{Typical values of extremal-weight combinatorial structures with independent symmetric weights}

\author{Yun Cheng}

\author{Yixue Liu}

\author{Tomasz Tkocz}

\author{Albert Xu}
\address{Carnegie Mellon University; Pittsburgh, PA 15213, USA.}
\email{$\{$yuncheng, ttkocz$\}$@andrew.cmu.edu, $\{$ayx, yixuel$\}$@alumni.cmu.edu}

\thanks{TT's research supported in part by NSF grant DMS-1955175.}

\date{\today}

\begin{document}

\begin{abstract} 
Suppose that the edges of a complete graph are assigned weights independently at random and we ask for the weight of the minimal-weight spanning tree, or perfect matching, or Hamiltonian cycle. For these and several other common optimisation problems, we establish asymptotically tight bounds when the weights are independent copies of a symmetric random variable (satisfying a mild condition on tail probabilities), in particular when the weights are Gaussian.
\end{abstract}

\maketitle

\bigskip

\begin{footnotesize}

\end{footnotesize}

\bigskip

\section{Introduction}

Classical optimisation problems such as the minimum spanning tree, the assignment problem or the shortest path have been extensively studied in the worst case as well as in the average case. For the latter, we usually consider a complete graph on $n$ vertices with each edge having assigned independently at random a nonnegative weight and ask for typical values (in terms of the expectation or high probability bounds) of the weight of the minimum-weight combinatorial structure such as a spanning tree, a perfect matching or a path between two fixed vertices. For example, in the case of the exponential rate $1$ weights, Frieze's $\zeta(3)$-result from \cite{F} says that the expected weight of the minimum spanning tree is asymptotic to $\zeta(3)$ as $n \to \infty$, for the assignment problem, resolving Parisi's conjecture from \cite{P},  Aldous in \cite{A} showed $\zeta(2)$ to be the asymptotic value (see also \cite{CS, LW, NPS, W}), whereas the weight of the shortest path is asymptotic to $\frac{\log n}{n}$, as showed independently by David and Prieditis in \cite{DP} and Janson in \cite{J}. Moreover, many limit theorems have also been established (for instance, see \cite{BH, FPS, J0}), large deviation regimes studied (for instance, see \cite{Fl, McD}) and various refinements, extensions and constrained versions have been investigated (for instance, see \cite{BLP, CFIJS, DF, FlFK, FPST, FPT, FS, FRT, FT1, FT3, FT4, FT2, K, LZ, LM, LT3}). 

In this short note, we consider the case of weights drawn from symmetric distributions. The assignment problem with Gaussian weights has been recently studied in {\red \cite{LT1, M} and with general weights in \cite{LT2}.} We identify a mild condition on tails, viz. Chernoff's bound being asymptotically optimal, notably satisfied in the Gaussian case. Under this condition, we are able to find asymptotically tight high probability estimates for the aforementioned optimisation problems (and several others).

\section{Results}

To motivate our main definition and state our results, we first recall necessary notions and facts. For a random variable $X$, let $\Lambda\colon \R \to (-\infty,+\infty]$,
\[
\Lambda(t) = \log\E e^{tX}, \qquad t \in \R,
\]
be its log-moment generating function with the Legendre transform (the rate function of $X$),
\[
\Lambda_*(t) = \sup_{s \in \R} \{st - \Lambda(s)\}, \qquad t \in \R.
\]
Recall that by Markov's inequality, for every $t$,
\[
\p{X > t} \leq \exp\Big(-\sup_{s > 0}\{st - \Lambda(s)\}\Big).
\]
If $\Lambda(t_0) < \infty$ for some $t_0 > 0$, then $\E X \in [-\infty,\infty)$ and for $t > \E X$, in fact we have $\sup_{s > 0}\{st - \Lambda(s)\} = \Lambda_*(t)$ (for instance, see Lemma 2.2.5 in \cite{DZ}). This is then sometimes referred to as Chernoff's inequality,
\begin{equation}\label{eq:cher}
\p{X > t} \leq \exp\left(-\Lambda_*(t)\right), \qquad t > \E X.
\end{equation}
We say that $X$ has \emph{regular upper tails} if this upper bound is asymptotically tight in the following sense:
\begin{equation}\label{eq:def-reg}
\p{X > t} = \exp\Big(-(1+o(1))\Lambda_*(t)\Big), \qquad \text{as } t \to +\infty.
\end{equation}
Note that when $X$ is bounded above, $\p{X > t} = 0$ and $\Lambda_*(t) = +\infty$ for all $t > \text{ess\! sup} X$, so in this case, condition \eqref{eq:def-reg} is vacuously satisfied and $X$ has regular upper tails. Examples with regular upper tails include Gaussian, exponential, gamma, Poisson random variables. On the other hand, it is not difficult to construct random variables without regular upper tails (see Section \ref{sec:conclusion}).

For the purposes of this note, we say that a random variable $X$ is \emph{good}, if it is symmetric (meaning $-X$ has the same distribution as $X$), $\Lambda(t) < \infty$ for all $|t| < \delta$ for some $\delta > 0$ and $X$ has regular upper tails. In particular, all symmetric bounded random variables are good. If $X$ is good and unbounded, then $\Lambda_* < \infty$ and $\Lambda_*$ is even, convex strictly increasing on $(0,+\infty)$ with $\Lambda_*(0) = 0$, so that the inverse $\Lambda_*^{-1}\colon [0,\infty) \to [0,\infty)$ of $\Lambda_*|_{[0,\infty)}$ is well-defined. {\red In general, for $t \geq 0$, we set $\Lambda_*^{-1}(t) = \inf\{s, \Lambda_*(s) \geq t\}$, the usual generalised inverse of $\Lambda_*$. For instance, in the degenerate case, when $X = 0$ a.s., $\Lambda \equiv 0$, and $\Lambda_*(0) = 0$ and $\Lambda_* = +\infty$ elsewhere, so $\Lambda_*^{-1} \equiv 0$.}

Examples of good random variables include of course standard Gaussian or two-sided exponential.

Our main results provide asymptotically tight high probability bounds on extremal-weight common combinatorial structures (perfect matchings, spanning trees, Hamilton cycles, paths between two fixed vertices, copies of a fixed graph) in complete graphs with edge-weights being i.i.d. copies of a good random variable. Thanks to symmetry, the distribution of the minimum is the same as of negative the maximum and thus we shall only focus on the latter.

\begin{theorem}\label{thm:assign}
Let $K_{n,n}  = ([n], [n], [n] \times [n])$ be the complete bipartite graph with each edge $e$ assigned an independent copy $X_e$ of a good random variable $X$ with rate function $\Lambda_*$. Let $\mathscr{C}_n$ be the set of perfect matchings in $K_{n,n}$ and let $W_n$ be the weight of an optimal matching,
\[
W_n = \max\left\{\sum_{e \in E(M)} X_e, \ M \in \mathscr{C}_n\right\}.
\]
Then,
\[
W_n = (1+o(1))n\Lambda_*^{-1}(\log n), \qquad \text{w.h.p.}\footnote{with high probability, that is with probability tending to $1$ as $n \to \infty$}
\]
\end{theorem}

Recall that the density of a graph $H=(V,E)$ is $d(H) = |E|/|V|$ and the graph $H$ is called \emph{balanced} if its density is not smaller than the density of any of its subgraphs, that is $\max d(H') = d(H)$ where the maximum is over all subgraphs $H'$ of $H$.

\begin{theorem}\label{thm:rest}
Let $K_{n}  = ([n], \binom{[n]}{2})$ be the complete (undirected) graph with each edge $e$ assigned an independent copy $X_e$ of a good random variable $X$ with rate function $\Lambda_*$. Given a set $\mathscr{C}_n$ of subgraphs of $K_n$, we let $W_n$ be the weight of an optimal one,
\[
W_n = \max\left\{\sum_{e \in E(H)} X_e, \ \text{$H \in \mathscr{C}_n$}\right\}.
\]
In each of the following cases
\begin{enumerate}[(a)]
\item $\mathscr{C}_n$ is the set of all spanning trees of $K_n$,

\item $\mathscr{C}_n$ is the set of all Hamilton cycles in $K_n$,

\item $\mathscr{C}_n$ is the set of all paths from vertex $1$ to $2$ in $K_n$,
\end{enumerate}
we have
\[
W_n = (1+o(1))n\Lambda_*^{-1}(\log n), \qquad \text{w.h.p.}
\]

\begin{enumerate}[(a)]

\item[(d)] If $\mathscr{C}_n$ is the set of all copies in $K_n$ of a fixed balanced graph $H_0$ with $\ell$ edges and density $d$, then
\[
W_n = (1+o(1))\ell\Lambda_*^{-1}\left(d^{-1}\log n\right), \qquad \text{w.h.p.}
\]
\end{enumerate}
\end{theorem}

{\red In the vast body of works mentioned in passing in the introduction, where the weights are nonnegative, it is natural to think of them as cost and those optimisation problems provide the size of the cheapest structure. Here, when the weights have a symmetric distribution, the paradigm is different and we may think of the random weight $X_e$ assigned to an edge $e$ as a gain from the edge (when positive) or, a loss (when negative, occurring equally likely, by symmetry). It then seems natural to ask for the maximal possible total gain over all structures, so in particular, $W_n$ above is the highest possible gain over all assignments (matchings), spanning trees, etc.}

{\red In the case of the optimal matching (the assignment problem), in their independent work \cite{LT2}, Lifshits and Tadevosian have recently obtained the asymptotics of $\E W_n$ for general i.i.d. weights whose quantile function tends to infinity and slowly varies at zero.}

In the important case of Gaussian weights, $\Lambda_*$ is explicit (quadratic). Moreover, the concentration for the supremum of a Gaussian process allows to obtain asymptotic values of the expectation as well. In the case of the optimal matching, the asymptotics of $\E W_n$ {\red was recently found by Mordant and Segers (see Theorem 2.3 in \cite{M}) and, independently, by Lifshits and Tadevosian (see Theorem 1 in \cite{LT1}), however with a different argument for the lower bound (via a greedy construction, whilst we employ pruning).}

\begin{corollary}\label{cor:gauss}
If in Theorem \ref{thm:assign} or \ref{thm:rest} we let the distribution of $X$ be standard Gaussian (mean $0$, variance $1$), then in Theorem \ref{thm:assign} as well as Theorem \ref{thm:rest} (a), (b), (c), we have
\[
W_n = (1+o(1))n\sqrt{2\log n}, \quad \text{w.h.p.} \quad \text{and} \quad \E W_n =(1+o(1))n\sqrt{2\log n},
\]
whilst in Theorem \ref{thm:rest} (d),
\[
W_n = (1+o(1))\ell\sqrt{2d^{-1}\log n}, \quad \text{w.h.p.} \quad \text{and} \quad \E W_n =(1+o(1))\ell\sqrt{2d^{-1}\log n}.
\]
\end{corollary}

\section{Proofs}

\subsection{Overview.}
Our upper bounds are based on the union bound. They turn out to be tight and are matched by lower bounds, obtained by means of constructions exploiting threshold probabilities for binomial random graphs.

\subsection{Regular tails.}
We shall need the following simple lemma which establishes asymptotic behaviour of certain sequences showing up in the proofs of our main results. The lemma is based on the tail regularity of good random variables.
\begin{lemma}\label{lm:xn}
Let $\alpha > 0$. Let $\omega_n$ be a positive sequence such that $\omega_n \to \infty$ with $\omega_n = n^{o(1)}$. Let $X$ be a good random variable and define $x_n = \inf\{t > 0, \ \p{X>t} \leq \omega_nn^{-\alpha}\}$. Then, {\red for all $n$ large enough, we have}
\[
(1-o(1))\Lambda_*^{-1}(\alpha\log n) \leq x_n \leq \Lambda_*^{-1}(\alpha \log n).
\]
\end{lemma}
\begin{proof}
{\red Excluding the degenerate situation, when $X = 0$ a.s. (in which case $x_n = 0$ and $\Lambda_*^{-1} \equiv 0$, so the lemma holds trivially), we can assume without loss of generality that $x_n > 0$ for all $n$ (since $\omega_nn^{-\alpha} \to 0$).  As $\omega_n \to \infty$, let us also assume that $\omega_n \geq 1$ for all $n$.}

By the definition of $x_n$, for every $\theta \in (0,1)$, we have $\p{X>\theta x_n} \geq \omega_nn^{-\alpha} \geq n^{-\alpha}$. On the other hand, by Chernoff's inequality \eqref{eq:cher}, $\p{X > \theta x_n} \leq e^{-\Lambda_*(\theta x_n)}$. Combining these two bounds yields $\Lambda_*(\theta x_n) \leq \alpha \log n$, so $x_n \leq \theta^{-1}\Lambda_*^{-1}(\alpha \log n)$, which proves the upper bound on $x_n$.

For the lower bound, if $X$ is bounded, say $A = \text{ess\! sup} |X|$, $0 < A < \infty$, then $x_n \to A$ as $n \to \infty$, as well as $\Lambda_*^{-1}(\alpha\log n) \to A$ as $n \to \infty$ (because $\Lambda_*(t) = +\infty$ for every $t > A$). Suppose now that $X$ is not bounded. Then $x_n \to \infty$ as $n \to \infty$. By assumption \eqref{eq:def-reg}, there is a positive sequence $\varepsilon_n \to 0$ such that $\p{X > x_n} = e^{-(1+\varepsilon_n)\Lambda_*(x_n)}$. By the definition of $x_n$, $\p{X > x_n} \leq \omega_nn^{-\alpha}$, thus $(1+\varepsilon_n)\Lambda_*(x_n) \geq -\log \omega_n + \alpha\log n$, so $\Lambda_*(x_n) \geq \theta_n\alpha\log n$ with $\theta_n = (1+\varepsilon_n)^{-1}(1-\frac{\log\omega_n}{\alpha\log n})$. By the assumptions on $\omega_n$, we have $\theta_n < 1$ with $\theta_n \to 1$, consequently $\theta_n\alpha\log n \geq \Lambda_*(\theta_n\Lambda_*^{-1}(\alpha\log n))$ (the convexity of $\Lambda_*$ and $\Lambda_*(0) = 0$ imply {\red that $u \mapsto \Lambda_*(u)/u$} is nondecreasing, which we shall also use several times in the sequel). We thus get $x_n \geq \theta_n\Lambda_*^{-1}(\alpha\log n) = (1-o(1))\Lambda_*^{-1}(\alpha\log n)$, as desired.
\end{proof}

\begin{remark}\label{rem:xn}
Thanks to Lemma \ref{lm:xn}, the asymptotic values of $W_n$ from Theorems \ref{thm:assign} and \ref{thm:rest} can be equivalently stated in terms of the sequences $x_n$, which may be easier to compute than $\Lambda_*^{-1}(\log n)$ for a given distribution of $X$.
\end{remark}

\subsection{Proof of Theorem \ref{thm:assign}}
We begin with a high probability upper bound on $W_n$.
Note that thanks to Chernoff's bound \eqref{eq:cher}, if $X_1, \dots, X_k$ are i.i.d. copies of $X$, then for every $t > 0$, 
\begin{equation}\label{eq:cher-sum}
\p{X_1+\dots+X_k > kt} \leq \exp\left\{-k\Lambda_*(t)\right\}
\end{equation}
(by independence, the rate function of $X_1+\dots+X_k$ at $kt$ is $k\Lambda_*(t)$).
Using first a union bound and then this, for $\delta_n \geq 0$, we obtain
\begin{align}\notag
\p{W_n > (1+\delta_n)n\Lambda_*^{-1}(\log n)} &\leq \sum_{M \in \mathscr{C}_n} \p{\sum_{e \in E(M)} X_e > (1+\delta_n)n\Lambda_*^{-1}(\log n)} \\\notag
&\leq \exp\Big\{\log|\mathscr{C}_n| - n\Lambda_*\left((1+\delta_n)\Lambda_*^{-1}(\log n)\right)\Big\} \\\label{eq:up-bd}
&\leq \exp\Big\{\log|\mathscr{C}_n|-(1+\delta_n)n\log n\Big\},
\end{align}
where in the last inequality we use the monotonicity of $\Lambda_*(u)/u$. Since $|\mathscr{C}_n| = n! \leq \frac{n^{n+1}}{e^n}$, for every $n \geq 7$, choosing $\delta_n = n^{-1}$ yields
\[
W_n \leq (1+n^{-1})n\Lambda_*^{-1}(\log n) \quad \text{with probability at least $1-e^{-n}$}.
\]

To establish a matching lower bound, we construct a random subgraph comprising only \emph{large} weights, which contains a perfect matching w.h.p. We set
\begin{equation}\label{eq:xn}
p_n = \frac{2\log n}{n} \qquad \text{and} \qquad x_n = \inf\{t > 0, \ \p{X>t} \leq p_n\}.
\end{equation}
{\red Excluding again the trivial case of $X = 0$ a.s., we have $x_n > 0$ (eventually), and then} for every $0 < \delta_n < 1$, we have $p_n' = \p{X > (1-\delta_n)x_n} > p_n$. Since the weights $X_e$ are i.i.d., the random bipartite graph $([n],[n],\{e \in [n]\times [n], \ X_e > (1-\delta_n)x_n\})$ is in fact $G_{n,n,p_n'}$, so by the classical result of Erd\"os and Renyi (see \cite{ER} or Theorem 6.1 in \cite{FK}), w.h.p. it contains a perfect matching which gives
\[
W_n \geq n(1-\delta_n)x_n, \qquad \text{w.h.p.}
\]
We choose $\delta_n$ arbitrarily with $\delta_n = o(1)$ as $n \to \infty$ and it remains to show that 
\[
x_n \geq (1-o(1))\Lambda_*^{-1}(\log n).
\]
This follows from Lemma \ref{lm:xn} applied to $\alpha = 1$ and $\omega_n = 2\log n$.\hfill$\square$

\subsection{Proof of Theorem \ref{thm:rest}}
We follow exactly the same strategy as in the proof of Theorem \ref{thm:assign}. 

\emph{Case (a), (b), (c): the upper bound.} In case (a) and (b) respectively, every graph $G \in \mathscr{C}_n$ has the same number of edges, $n-1$ and $n$ respectively. Repeating verbatim the derivation of \eqref{eq:up-bd} yields in each case
\[
\p{W_n > (1+\delta_n)n\Lambda_*^{-1}(\log n)} \leq \exp\Big\{\log|\mathscr{C}_n|-(1+\delta_n)n\log n\Big\}.
\]
The same holds in case (c) because letting $\ell(P)$ be the number of edges on a path $P$, using \eqref{eq:cher-sum}, we have
\begin{align*}
\p{W_n > (1+\delta_n)n\Lambda_*^{-1}(\log n)} &\leq \sum_{P \in \mathscr{C}_n} \p{\sum_{e \in E(P)} X_e > (1+\delta_n)n\Lambda_*^{-1}(\log n)} \\
&\leq \sum_{P \in \mathscr{C}_n}\exp\left\{- \ell(P)\Lambda_*\left((1+\delta_n)\frac{n}{\ell(P)}\Lambda_*^{-1}(\log n)\right)\right\} \\\label{eq:up-bd}
&\leq \sum_{P \in \mathscr{C}_n}\exp\Big\{- (1+\delta_n)n\log n\Big\},
\end{align*}
where in the last inequality we use the convexity of $\Lambda_*$ and $(1+\delta_n)\frac{n}{\ell(P)} \geq 1$. In case (a), (b), (c), we have $|\mathscr{C}_n| = n^{n-2}$, $|\mathscr{C}_n| = \frac12(n-1)!$, $|\mathscr{C}_n| = \sum_{l=0}^{n-2}\binom{n-2}{l}l! = (n-2)!\sum_{l=0}^{n-2}\frac{1}{(n-2-l)!} \leq e(n-2)!$, respectively, so in each case, setting $\delta_n = 0$ suffices to get $\log|\mathscr{C}_n|-(1+\delta_n)n\log n \to -\infty$ and as a result,
\[
W_n \leq n\Lambda_*^{-1}(\log n), \qquad \text{w.h.p.}
\]

\emph{Case (a), (b), (c): the lower bound.} We define $p_n, p_n'$ and $x_n$ as in the proof of Theorem~\ref{thm:assign}, see \eqref{eq:xn}. Considering the random graph $\mathcal{G} = ([n], \{e \in \binom{[n]}{2}, \ X_e > (1-\delta_n)x_n\})$ which is distributed as $G_{n,p_n'}$, we get in each case that $W_n \geq {\red (n-1)}(1-\delta_n)x_n$ w.h.p. This is because $p_n' > \frac{2\log n}{n}$ guarantees that $\mathcal{G}$ is w.h.p. (a) connected, (b) Hamiltonian, (c) Hamiltonian-connected (in particular $\mathcal{G}$ has a path of length $n$ between vertices $1$ and $2$), see \cite{ER0}, \cite{KSz}, \cite{BFF} respectively, or Theorems 4.1, 6.5 and Exercise 6.7.11 in \cite{FK}. Then Lemma \ref{lm:xn} finishes the argument.

\emph{Case (d): the upper bound.} Let $v$ and $\ell$ denote the number of vertices and edges in $H_0$, respectively. The density of $H_0$ is $d = \frac{\ell}{v}$. For the cardinality $|\mathscr{C}_n|$, that is the number of copies of $H_0$ in $K_n$, we have $|\mathscr{C}_n| \leq \binom{n}{v}v! \leq n^v$. As in \eqref{eq:up-bd}, we obtain
\begin{align*}
\p{W_n > (1+\delta_n)\ell\Lambda_*^{-1}(d^{-1}\log n)} &\leq \exp\left\{\log|\mathscr{C}_n|- \ell\Lambda_*\left((1+\delta_n)\Lambda_*^{-1}(d^{-1}\log n)\right)\right\} \\
&\leq \exp\left\{v\log n - (1+\delta_n)\ell d^{-1}\log n\right\} \\
&=  \exp\left\{-\delta_n v\log n\right\}.
\end{align*}
Choosing, say $\delta_n = (\log n)^{-1/2}$, we conclude
\[
W_n \leq (1+\delta_n)\ell\Lambda_*^{-1}(d^{-1}\log n), \qquad \text{w.h.p.}
\]

\emph{Case (d): the lower bound.} It is a classical result of Erd\"os and R\'enyi from \cite{ER1} (see also Theorem 5.3 in \cite{FK}) that for a balanced graph $H_0$, when $pn^{1/d} \to \infty$, the random graph $G_{n,p}$ contains a copy of $H_0$ w.h.p. Therefore, letting $p_n = \omega_nn^{-1/d}$ with $\omega_n \to \infty$, defining $x_n$ as in \eqref{eq:xn} and considering the random graph $\mathcal{G}$, we obtain as in cases (a), (b), (c), $W_n \geq \ell(1-\delta_n)x_n$ w.h.p. It remains to show that $x_n \geq (1-o(1))\Lambda_*^{-1}(d^{-1}\log n)$. This follows from Lemma \ref{lm:xn} applied to $\alpha = d^{-1}$ as long as we choose $\omega_n \to \infty$ with $\omega_n = n^{o(1)}$. \hfill$\square$

\begin{remark}
Using the common approach based on moment generating functions (for example, see the proof of (A.3) in Appendix A.2 in \cite{Ch}), we can also obtain the following upper bound on the expectation in Theorems \ref{thm:assign} and \ref{thm:rest},
\begin{equation}\label{eq:EWn}
\E W_n \leq l \Lambda_*^{-1}(l^{-1}\log |\mathscr{C}_n|),
\end{equation}
where $l = \max_{H \in \mathscr{C}_n} |E(H)|$ is the maximal number of edges in the graphs from a given class $\mathscr{C}_n$ (so $l = n, n-1, n, n, \ell$ in Theorems \ref{thm:assign} and \ref{thm:rest} (a), (b), (c), (d), respectively). Moreover, in each of the cases, the right hand side is asymptotic to the high probability bound on $W_n$ from Theorems \ref{thm:assign} and \ref{thm:rest}. A proof of \eqref{eq:EWn} can be sketched as follows: using $\max x_k \leq \frac{1}{t}\log \sum_{k} e^{t x_k}$ valid for all $t > 0$ and $x_k \in \R$, the concavity of the $\log$ function and independence, we get
\[
\E W_n \leq \frac{1}{t}\log\sum_{H \in \mathscr{C}_n} \E e^{t\sum_{e \in E(H)} X_e}  = \frac{1}{t}\log\sum_{H \in \mathscr{C}_n} e^{|E(H)|\Lambda(t)} \leq \frac{\log|\mathscr{C}_n| + l\Lambda(t)}{t}.
\]
Taking the infimum over $t > 0$ finishes the argument.
\end{remark}

\subsection{Proof of Corollary \ref{cor:gauss}}
When $X$ is standard Gaussian, $\E e^{tX} = e^{t^2/2}$, so $\Lambda_*(t) = \frac{t^2}{2}$ and $\Lambda_*^{-1}(t) = \sqrt{2t}$. If we let $l = \max_{H \in \mathscr{C}_n} |E(H)|$, then the variance of each Gaussian that $W_n$ takes the maximum over is bounded by $l$, $\Var(\sum_{e \in E(H)} X_e) = |E(H)| \leq l$. From the concentration of the maximum of a Gaussian process  around its expectation (see for instance Theorem 7.1 in \cite{L}), we get
\begin{equation}\label{eq:conc}
\p{|W_n - \E W_n| \geq t} \leq 2e^{-t^2/(2l)}, \qquad t \geq 0.
\end{equation}
In the case of Theorem \ref{thm:assign}, $l = n$, so taking, say $t = n$, we get from the above that $|W_n - \E W_n| < n$ w.h.p. Combining this with Theorem \ref{thm:assign} which gives $W_n = (1+o(1))n\sqrt{2\log n}$ w.h.p., we obtain $\E W_n = (1+o(1))n\sqrt{2\log n}$, as desired. We proceed analogously in the case of Theorem \ref{thm:rest} (and omit the details).\hfill$\square$

\section{Final remarks}\label{sec:conclusion}

\subsection{Conclusion}
We have determined asymptotically typical values of the weight of the minimal-weight common combinatorial structures in complete graphs with independent identically distributed symmetric weights having regular tails, that is satisfying \eqref{eq:def-reg}. 
A natural next step would be to establish limit theorems and the order of fluctuations.

{\red 
\subsection{Fluctuations}
Even in the case of Gaussian weights, where many tools exist (e.g. \cite{Ch, Ch2, DEZ}), finding the asymptotic value of the variance seems interesting and challenging. In particular, for optimal matchings, i.e. in the setting of Theorem \ref{thm:assign}, it follows from \eqref{eq:conc} that $\Var(W_n) \leq 4n$. On the other hand, Mordant and Segers argued in \cite{M} that $\Var(W_n) \geq 1$ (exploiting symmetries of the covariance structure). They also noted that by a general phenomenon in superconcentration (Theorem 8.1 of Chatterjee from \cite{Ch}) and the asymptotics of the mean $\E W_n$, the upper bound in fact improves and we have $\Var(W_n) = o(n)$, however the exact order of fluctuations $\Var(W_n)$ seems elusive.
}

\subsection{Refinements}
It is instructive to see the shortcoming of our main result when applied to nonnegative weights.  Suppose we consider the minimum spanning tree problem on the complete graph with independent weights $Y_e$, each uniformly distributed on $[0,2]$. Since $X_e = Y_e - 1$ is uniform on $[-1,1]$, applying Theorem \ref{thm:rest} (and Remark \ref{rem:xn}) with $x_n = \inf\{t > 0, \ \p{X > t} \leq \frac{2\log n}{n}\} = 1 - \frac{4\log n}{n}$, w.h.p., we have
\[
-(1+o(1))n\left(1-\frac{4\log n}{n}\right) = W_n = \min_{T} \sum_{e \in E(T)} X_e = \left(\min_{T} \sum_{e \in E(T)} Y_e\right) - (n-1).
\]
Without knowing the implicit $o(1)$ term, we cannot infer the asympototic behaviour of $\min_{T} \sum_{e \in E(T)} Y_e$ (which by Frieze's result from \cite{F}, tends to $2\zeta(3)$ in probability). It would be of interest to remedy this and refine the $o(1)$ term for (general) symmetric distributions.

\subsection{Measures without regular tails}
Recall that a random variable $X$ is said to have regular upper tails if $\lim_{t\to\infty} \frac{-\log \p{X>t}}{\Lambda_*(t)} = 1$, {\red see \eqref{eq:def-reg}}. We show an example of $X$ for which this does not hold. The idea of our example comes from \cite{GG}. Fix an increasing sequence of positive numbers $0 = x_0 < x_1 < x_2 < \dots$ such that $x_n \to \infty$ and let $\log 2 = y_0 < y_1 < y_2 < \dots$ be an increasing sequence of positive numbers such that $y_n \to \infty$. Define the following nonincreasing right-continuous step function
\[
T(t) = \sum_{n=0}^\infty e^{-y_n}\1_{[x_n,x_{n+1})}(t), \qquad t \geq 0.
\]
Let $X$ be a symmetric random variable such that $\p{X > t} = T(t)$, $t \geq 0$. In other words, $X$ is discrete taking the values $\pm x_n$ with probabilities $e^{-y_{n-1}}-e^{-y_n}$, $n = 1, 2, \dots$. Choosing $y_n$ growing \emph{much faster} relative to $x_n$, it is easily guaranteed that $\Lambda(t) < \infty$ for every $t \in \R$ (it suffices that $\sum e^{tx_n - y_{n-1}} < \infty$ for every $t >0$). We show that we can choose the $y_n$ such that $\limsup_{t\to\infty} \frac{-\log\p{X>t}}{\Lambda_*(t)} = \infty$ and consequently, $X$ does not have regular upper tails and hence is not \emph{good} {\red (in the sense of \eqref{eq:def-reg} of our definition)}. By Chernoff's inequality, $e^{-y_{n-1}} = \p{X \geq x_n} \leq e^{-\Lambda_*(x_n)}$, so $\Lambda_*(x_n) \leq y_{n-1}$ for every $n = 1, 2, \dots$. Therefore, by the convexity of $\Lambda_*$, for $\alpha \in (0,1)$,
\[
\Lambda_*(\alpha x_n + (1-\alpha)x_{n+1}) \leq \alpha y_{n-1} + (1-\alpha)y_n
\]
and we obtain
\begin{align*}
\frac{-\log \p{X > \alpha x_n + (1-\alpha)x_{n+1}}}{\Lambda_*(\alpha x_n + (1-\alpha)x_{n+1})} &= \frac{y_n}{\Lambda_*(\alpha x_n + (1-\alpha)x_{n+1})} \\
&\geq \frac{y_n}{\alpha y_{n-1} + (1-\alpha)y_n}.
\end{align*}
If we choose the sequence $(y_n)$ to grow fast enough, specifically such that $\frac{y_{n-1}}{y_n} \to 0$ (say $y_n = 2^{n^2}$, $n \geq 1$), then we get $\limsup_{t\to\infty} \frac{-\log\p{X>t}}{\Lambda_*(t)} \geq \frac{1}{1-\alpha}$. Letting $\alpha \to 1$ finishes the argument.

\subsection*{Acknowledgements.} 
We would like to thank Alan Frieze for the many illuminating discussions. We are indebted to Mikhail Lifshits for the helpful correspondence. We should like to thank anonymous referees for their very useful reports considerably improving our manuscript. Finally, we are grateful for the excellent working conditions provided by the undergraduate programme ``2020 Summer Experiences in Mathematical Sciences'' at Carnegie Mellon University.

\end{document}